\documentclass[journal=JOLT,lang=english,biblatex]{mersenne}
% ...

\usepackage{algorithm,algorithmic}

\newcommand{\dd}{\textnormal{d}}

\newcommand{\ad}{\textnormal{ad}}

\newcommand{\R}{\mathbb{R}}

\newcommand{\Lt}{\mathrm{L}}
\newcommand{\Rt}{\mathrm{R}}
\newcommand{\Tt}{\mathrm{T}}

\newcommand{\BCH}{\textnormal{BCH}}
\newcommand{\mbg}{\mathbf g}
\newcommand{\mbv}{\mathbf v}
\newcommand{\mbf}{\mathbf f}
\newcommand{\mbt}{\mathbf t}
\newcommand{\mbx}{\mathbf x}
\newcommand{\mby}{\mathbf y}
\newcommand{\mbr}{\mathbf r}
\newcommand{\mbz}{\mathbf z}
\newcommand{\mbthetabi}{\mathbf{\vartheta}^{\textnormal{BI}}}

\DeclareMathOperator{\Log}{Log}
\DeclareMathOperator{\Exp}{Exp}

\title{Bi-invariant Geodesic Regression: \\ Existence, Uniqueness, and Convergence}

\addauthor[
address = {Zuse Institute Berlin, Berlin Germany},
email = {hanik@zib.de},
corresponding,
]
{\firstname{Martin} \lastname{Hanik}}

\addauthor[
address = {Zuse Institute Berlin, Berlin Germany},
email = {vontycowicz@zib.de},
]{\firstname{Christoph} \vonname{von}\nobreakauthor\lastname{Tycowicz}}

\fundings{M.\@ Hanik is funded by the Deutsche Forschungsgemeinschaft (DFG, German Research
Foundation) under Germany's Excellence Strategy – The Berlin Mathematics Research Center MATH+ (EXC-2046/2, project ID: 390685689) \MRSgrant{EXC-2046/2, project ID: 390685689}}



\begin{abstract}
Bi-invariant geodesic regression generalizes linear regression to Lie groups. Its main feature is that it respects the symmetries of the group so that the resulting estimator is independent of arbitrary choices such as a reference frame. However, the local existence and uniqueness of the underlying estimator have not been shown until now. Furthermore, the convergence properties of the proposed algorithm for computing the estimator are not known. In this work, we investigate these questions. We prove that, locally, a unique estimator exists and give explicit bounds on the size of this neighborhood. We also show that the proposed iterative algorithm converges linearly to this estimator.
\end{abstract}

\begin{document}
\maketitle

\section{Introduction}
    Manifold-valued data arises in many applications and can often be modeled as elements of a Lie group. Examples include representations of skeletal systems in robotics~\cite{ParkBobrowPloen1995}, position- and motion-independent recognition of objects in computer vision~\cite{vemulapalli2014human,veeriah2015differential,vemulapalli2016rolling,huang2017deep}, and medical imaging~\cite{Boisvert_ea2008,vemulapalli2014human,hou2018computing,AmbellanZachowvonTycowicz2019_GL3,Pennec_ea2019_book}. Lie groups provide a natural framework for modeling continuous transformations and, in many applications, the group structure carries important geometric information that should be preserved by statistical methods.

However, using Lie groups for data representation typically involves arbitrary choices, such as the choice of a coordinate system or reference frame. To obtain statistical results that are independent of such choices, the underlying methods should respect the symmetries of the group, namely left and right translations and inversion~\cite{Hanik2022,lawson2026weighted,PennecLorenzi2020}. In particular, an estimator should transform consistently when the data are transformed by one of these symmetries. This property is commonly expressed in terms of equivariance.

Motivated by this principle, Schade, von Tycowicz, and Hanik~\cite{Schade2025Biinvariant} proposed Bi-invariant Geodesic Regression, a generalization of linear regression to Lie groups. The method is based on geodesics induced by a bi-invariant connection and is designed to preserve the symmetries of the underlying Lie group. In particular, the authors showed that the resulting estimator is equivariant under left and right translations 
%%% Include again when MedIA article is out %%%
%\textcolor{red}{and inversion}. 
They also proposed an iterative algorithm for computing the estimator. Thus, Bi-invariant Geodesic Regression provides a regression framework that combines the geometric structure of Lie groups with the symmetry properties required for faithful statistical analysis.

The theoretical properties of the estimator, however, were left open. In particular, it was not known under which conditions the estimator exists and is unique. Moreover, although an iterative algorithm was proposed for its computation, no convergence results were established. These questions are important both for the mathematical foundations of the method and for understanding the behavior of the numerical procedure.

In this work, we address these questions. Using a fixed point approach, we show that, for sufficiently small data, there exists a unique local estimator and derive explicit bounds on the size of the neighborhood in which existence and uniqueness are guaranteed. Furthermore, we prove that the proposed iterative algorithm is a contraction on this neighborhood and consequently converges linearly to the unique estimator. The local results can then be transferred to arbitrary reference points using the equivariance of the estimator and the iteration. 

The remainder of the paper is organized as follows. In Section~\ref{sec:background}, we recall the necessary background on Bi-invariant Geodesic Regression and its iterative estimation algorithm, and establish the equivariance properties of the estimator and the iterates. Section~\ref{sec:preliminaries} develops the local Lie-algebra formulation of the iteration and derives the expansion that forms the basis of our fixed point analysis. In Section~\ref{sec:existence}, we use this expansion to prove local existence and uniqueness of the estimator and establish linear convergence of the iterative algorithm. We conclude by discussing how the local results extend to data in a neighborhood of an arbitrary point through equivariance.

\section{Background and Problem Setting} \label{sec:background}

\subsection{Bi-invariant Geodesic Regression} \label{sec:biinvariant_regression}

Let $G$ be a finite-dimensional Lie group endowed with its canonical Cartan-Shouten (CCS) connection. Let $U \subset G$ be a normal-convex neighborhood; that is, any pair of points $f,g \in U$ can be connected by a unique geodesic $[0,1] \ni t \mapsto \gamma(t; f, g)$ that never leaves $U$. Every point in $G$ has such a neighborhood~\cite{Postnikov2013}. The exponential and logarithm of the CCS connection at $g \in U$ are given by~\cite[Corollary 5.1]{PennecLorenzi2020} 
    \begin{align} 
            \Exp_g(v) &= g \exp \left(\dd_g \Lt_{g^{-1}}(v) \right) = \exp \left(\dd_g \Rt_{g^{-1}}(v) \right) g, \quad v \in T_gG, \label{eq:EXP}\\ 
            \Log_g(f) &= \dd_e\Lt_g \big(\log(g^{-1}f)\big) = \dd_e\Rt_g \big(\log(fg^{-1})\big), \quad f \in U. \nonumber %\label{eq:LOG}
        \end{align}
The geodesics are translated one-parameter subgroups, that is, for all $t \in [0,1]$,
\begin{align}
\begin{split}\label{eq:geodesic}
            \gamma(t; g, f) &= \Exp_g \big(t \Log_g(f) \big) \\
            &= g \exp \big( t\log(g^{-1}f) \big) \\
            &= \exp \big( t\log(fg^{-1}) \big) g.
        \end{split}
    \end{align}

For a $G$-valued random variable $Y$ and a non-random, real-valued variable $t$, the generalization~\cite{Fletcher2013} of the linear regression model to the manifold setting is the geodesic model
\begin{equation} \label{eq:model}
    Y = \Exp_{\gamma(t;g_0, g_1)}\bigl(\epsilon \bigr),
\end{equation}
where $\epsilon$ is a random variable taking values in the tangent space at $\gamma(t; g_0, g_1)$. The points $g_0, g_1 \in G$ are the model's parameters that must be estimated from a data set $(f_1, t_1), \dots, (f_N, t_N) \in U \times [0,1]$; we denote such data  in short by $(\mbf,\mbt)$.

From now on, the parameter $t$ of $\gamma$ often becomes a subscript when the boundary parameters are what we are interested in.

In \cite{Schade2025Biinvariant}, Schade, von Tycowicz, and Hanik propose the bi-invariant estimator.
\begin{definition} \label{def:bi_inv_estiamtor}
Let $U \subseteq G$ be a normal convex neighborhood. Furthermore, let $(\mbf,\mbt) \in U^N \times [0,1]^N$. The \textit{bi-invariant estimator} for (\ref{eq:model}) is the geodesic between the points $\mbthetabi(\mbf,\mbt):= (\hat g_0,\hat g_1) \in U^2$ that satisfy\footnote{We denote partial derivatives with respect to the $j$-th component by $\partial_j$.}
\begin{align}
    \sum_{i=1}^N(1-t_i)^2 \bigl( \partial_0 \gamma_{t_i}(\hat g_0,\hat g_1) \bigr)^{-1} \left (\Log_{\gamma_{t_i}(\hat g_0,\hat g_1)}(f_i) \right)&= 0 \qquad\textnormal{and} \label{eq:bir1} \\
    \sum_{i =1}^N t_i^2 \bigl( \partial_1 \gamma_{t_i}(\hat g_0,\hat g_1) \bigr)^{-1} \left( \Log_{\gamma_{t_i}(\hat g_0,\hat g_1)}(f_i) \right) &= 0. \label{eq:bir2}
\end{align}
\end{definition}

\subsection{Iterative Estimation Algorithm}

Whenever we encounter $G^k$, $k \ge 2$, we use the standard product structure and product connection such that the group operation and the connection with its associated maps operate entirely component-wise across each factor. This is true, for example, for $\Exp / \exp$, $\Log / \log$, and the left and right translations $\Lt$ and $\Rt$. For $k\geq 1$, we use the same notation for the component-wise extensions of $\Exp$, $\exp$, $\Log$, and $\log$; for example,
$$
\exp:(T_eG)^k\to G^k,\qquad
\exp(x_1,\ldots,x_k)
:=
(\exp x_1,\ldots,\exp x_k),
$$
and, whenever defined,
$$
\log:G^k\to(T_eG)^k,\qquad
\log(g_1,\ldots,g_k)
:=
(\log g_1,\ldots,\log g_k).
$$
Thus, the meaning of each map is determined by the domain and codomain.
Furthermore, for $\Tt \in \{\Lt, \Rt \}$, $h\in G$ and $\mbx=(x_1,\ldots,x_k)\in G^k$, we use the notation
$$
\Tt_h(\mbx)
:=
\bigl(\Tt_h(x_1),\ldots,\Tt_h(x_k)\bigr),
$$
that is, $\Tt_h$ acts component-wise on tuples.
\bigskip

Setting $b_i=\gamma_{t_i}(g_0,g_1)$ for $i=1,\dots,N$ and writing $\mbg = (g_0, g_1)$ for elements in $U^2$, define
\begin{align} \label{eq:v}
        \mbv \left(\mbg, \mbf, \mbt \right) &:= \begin{pmatrix}
        \sum_{i=1}^N (1-t_i)^2 \bigl( \partial_0 \gamma_{t_i}(g_0, g_1) \bigr)^{-1} \left( \Log_{b_i} (f_i)\right), \\ 
        \sum_{i=1}^N t_i^2 \bigl( \partial_1 \gamma_{t_i}(g_0, g_1) \bigr)^{-1} \left( \Log_{b_i} (f_i)\right)
        \end{pmatrix} \in T_{\mbg}G^2.
\end{align}
The components of $\mbv$ are the left-hand sides of \eqref{eq:bir1} and \eqref{eq:bir2}; they act as a net force, which the data points apply (through the geodesic) to $g_0$ and $g_1$. Schade et al.\ propose the following algorithm to iteratively compute the estimator: Starting with the initial guess $\mbg_0 := (g_0, g_1) \in U^2$, we compute the iterates
\begin{align} \label{alg:algorithm}
        \mbg_{n+1} := \Exp_{\mbg_n} \biggl( \lambda\, \mbv \Bigl(\mbg_n, \mbf, \mbt \Bigr) \biggr), \qquad n=1,2,\dots
\end{align}
with a small enough step size $\lambda > 0$ until convergence.

An important feature of the algorithm is that it is equivariant under right and left translations. To show this, we first establish the equivariance of the Cartan--Schouten exponential.
\begin{lemma}
\label{lem:equivariance_exp}
For
$$
\Tt\in\{\Lt,\Rt\},\qquad h,g\in G,\qquad v\in T_gG,
$$
we have
$$
\Exp_{\Tt_h(g)}
\bigl(\dd_g \Tt_h(v)\bigr)
=
\Tt_h\bigl(\Exp_g(v)\bigr).
$$
\end{lemma}
\begin{proof}
We start with the case $\Tt = \Lt$.
We have
$$
\Lt_{hg}=\Lt_h\circ \Lt_g
$$
and
$$
\Lt_{(hg)^{-1}}\circ \Lt_h
=
\Lt_{g^{-1}h^{-1}}\circ \Lt_h
=
\Lt_{g^{-1}}.
$$
Hence, by the chain rule,

$$
\dd_{hg}\Lt_{(hg)^{-1}}\circ\dd_g\Lt_h
=
\dd_g\Lt_{g^{-1}},
$$
and therefore
$$
\begin{aligned}
\Exp_{hg}\bigl(\dd_g\Lt_h(v)\bigr)
&=
hg\exp\left(
\dd_{hg}\Lt_{(hg)^{-1}}
\bigl(\dd_g\Lt_h(v)\bigr)
\right)\\
&=
hg\exp\left(\dd_g\Lt_{g^{-1}}(v)\right)\\
&=
h\Exp_g(v).
\end{aligned}
$$

The proof works analogously for right translations.
\end{proof}

Having established the equivariance of $\Exp$, the same can be shown for the algorithm.
\begin{proposition}[Translation-equivariance of the algorithm]
\label{prop:equivariance_algorithm}
Let data $f_1,\dots,f_N\in U$ and
$$
\tilde f_i:= \Tt_{h}(f_i), \qquad \Tt \in \{\Lt, \Rt\},
$$ for some $h \in G$.
Let further $\mbg_n$ and $\tilde \mbg_n$ be the iterates~\eqref{alg:algorithm} for the original and translated data, respectively. If the two initial iterates are $g_0$ and $\Tt_{h}(g_0)$, then
$$
\tilde \mbg_n = \Tt_h (\mbg_n)
$$
for all $n$ for which they are well defined.
\end{proposition}
\begin{proof}
The proof proceeds by induction over $n$. By assumption,
$$
\tilde{\mbg}_0=\Tt_h(\mbg_0),
$$
so the assertion holds for $n=0$.

Suppose that
$$
\tilde{\mbg}_n=\Tt_h(\mbg_n)
$$
for some $n$. The update vector $\mbv$ is equivariant:
$$
\mbv\bigl(\Tt_h(\mbg_n),\Tt_h(\mbf),\mbt\bigr)
=
\dd_{\mbg_n}\Tt_h
\bigl(
\mbv(\mbg_n,\mbf,\mbt)
\bigr);
$$
see the proof of~\cite[Theorem 1]{Schade2025Biinvariant} for a derivation.

Consequently,
$$
\begin{aligned}
\tilde{\mbg}_{n+1}
&=
\Exp_{\tilde{\mbg}_n}
\left(
\lambda\,
\mbv(\tilde{\mbg}_n,\Tt_h(\mbf),\mbt)
\right)
\\
&=
\Exp_{\Tt_h(\mbg_n)}
\left(
\lambda\,
\dd_{\mbg_n}\Tt_h
\bigl(
\mbv(\mbg_n,\mbf,\mbt)
\bigr)
\right)
\\
&=
\Tt_h
\left(
\Exp_{\mbg_n}
\left(
\lambda\,
\mbv(\mbg_n,\mbf,\mbt)
\right)
\right)
\\
&=
\Tt_h(\mbg_{n+1}),
\end{aligned}
$$

where the third equality follows from Lemma~\ref{lem:equivariance_exp} because $\Exp$ works component-wise. Thus the assertion holds for $n+1$, completing the induction.
\end{proof}  

\subsection{Equivariance of the Estimator}

Schade et al.\ show that this estimator is equivariant under left/right translations.
\begin{theorem}[Translation-equivariance of the estimator] \label{thm:main_trans}
    Let $(\mbf,\mbt) \in U^N \times [0,1]^N$. Then, for all $h \in G$,
    \begin{align*}
        \mbthetabi \left( \Tt_h(\mbf), \mbt \right) &= \Tt_h \bigl(\mbthetabi (\mbf, \mbt) \bigr), \qquad \Tt_h \in \{ \Lt_h, \Rt_h \}.
    \end{align*} 
\end{theorem}
In the coming sections, we investigate the fixed point properties of the sequence~\eqref{alg:algorithm}. Since the estimator and the sequence to compute it (Proposition~\ref{prop:equivariance_algorithm}) are equivariant under translations, we can thereby focus on data in a neighborhood of the identity. Once we prove existence, uniqueness, and convergence for this scenario, it can be easily translated to the general case.

\section{Preliminaries and Local Expansion} \label{sec:preliminaries}

In this, section we investigate the local update map that the algorithm induces. To this end, necessary Lie algebra preliminaries are established first. In the following we can often choose whether we express maps (for example, the geodesic map $\gamma$) with left or right translations. We always pick the former convention. 

\subsection{Lie Algebra Preliminaries}
 
Let $x, y\in T_eG$. Recall that $\ad_x(y) = [x,y]$. We define 
\begin{equation} \label{eq:ad_m}
    \ad_x^n(y) := \ad_x \circ \dots \circ \ad_x (y) = [x[x[\dots [x[x,y]]\dots]]],
\end{equation}
where $\ad_x$ and $x$ appear $n$ times in the middle and on the right-hand side, respectively. Clearly, 
\begin{equation} \label{eq:ad_dist}
    \ad_x^n \circ \ad_x^k = \ad_x^{n+k}.
\end{equation}

We will need the fact~\cite[Section~1.5]{duistermaat2012lie} that the differential of the exponential $\dd_x \exp: T_xT_eG \cong T_eG \to T_{\exp(x)}G$ is given by
\begin{equation} \label{eq:derivative_exp}
    \dd_x \exp = \dd_e \Lt_{\exp(x)} \circ \sum_{n=0}^\infty \frac{(-1)^n \ad_x^n}{(n+1)!}.
\end{equation} 
Remember the Bernoulli numbers are given by
\begin{align*}
    B_0 &:= 1, \\
    B_n &:= -\frac{1}{n+1} \sum_{k=0}^{n-1} \binom{n+1}{k} B_k, \quad n = 1,2,\dots.
\end{align*}
The differential of the group logarithm ~\cite[Section~1.6]{duistermaat2012lie} can be expressed as
\begin{equation} \label{eq:derivative_log}
    \dd_{g} \log = \sum_{n=0}^\infty \frac{(-1)^n B_n \ad_{\log(g)}^n}{n!} \circ \dd_g \Lt_{g^{-1}}
\end{equation}
as long as $\rho(\ad_{\log(g)}) := \textnormal{max}\{|\mu|: \mu \textnormal{ is an eigenvalue of } \ad_{\log(g)} \} < 2 \pi$.

The following two results yield a local expression for the inverted geodesic operator differentials appearing in $\mathbf{v}$.
\begin{lemma} \label{lem:diff_2_gam}
    Let $U \subseteq G$ be a normal convex neighborhood and $f,g \in U$. Let further $u \in T_{\gamma(t;g,f)}G$ and $\bar u := \dd_{\gamma(t;g,f)}L_{\gamma(t;g,f)^{-1}}(u)$ its representative in $T_eG$ that is obtained by left translation. With this, the endpoint derivative of a Cartan-Shouten geodesic is given by
    \begin{equation*}
        \bigl( \partial_1 \gamma_t (g,f) \bigr)^{-1} (u)
        = \sum_{m=0}^\infty C_m(t) \dd_e \Lt_f \bigl( \ad_{\log(g^{-1}f)}^m(\bar{u}) \bigr), 
    \end{equation*}
where 
$$C_m(s) := (-1)^{m} \sum_{k=0}^m \frac{s^{k-1} B_k}{(m-k+1)!k!},$$
whenever $\rho(\ad_{\log{(g^{-1}f})}) < 2\pi$.
\end{lemma}
\begin{proof}
    Note that
    \begin{equation*}
        \gamma_{t,g} = \Lt_g \circ \exp \circ\, t\log \circ\, \Lt_{g^{-1}}.
    \end{equation*}
    Thus, due to the chain rule,
    \begin{equation*}
        \partial_1 \gamma_t (g,f) = \dd_{\exp(t\log(g^{-1}f))} \Lt_g \circ \dd_{t\log(g^{-1}f)} \exp \circ\, t\, \dd_{g^{-1}f}\log \circ\, \dd_f\Lt_{g^{-1}}.
    \end{equation*}
    It follows that
    \begin{align*}
        \bigl( \partial_1 \gamma_t (g,f) \bigr)^{-1} &= \bigl( \dd_{\exp(t\log(g^{-1}f))} \Lt_g \circ \dd_{t\log(g^{-1}f)} \exp \circ\, t\, \dd_{g^{-1}f}\log \circ\, \dd_f\Lt_{g^{-1}} \bigr)^{-1} \\
        &= \bigl( \dd_f\Lt_{g^{-1}} \bigr)^{-1} \,\circ\, \frac{1}{t}\, \bigl( \dd_{g^{-1}f}\log \bigr)^{-1}  \\ 
        &\qquad \qquad \circ \bigl( \dd_{t\log(g^{-1}f)} \exp \bigr)^{-1} \,\circ\, \bigl( \dd_{\exp(t\log(g^{-1}f))} \Lt_g \bigr)^{-1} \\
        &= \dd_{g^{-1}f}\Lt_g \,\circ\, \frac{1}{t}\, \dd_{\log(g^{-1}f)}\exp \\ 
        &\qquad \qquad \circ\, \dd_{\exp(t\log(g^{-1}f))} \log \,\circ\, \dd_{g\exp(t\log(g^{-1}f))} \Lt_{g^{-1}}.
    \end{align*}
    
    When $\rho(\ad_{\log{(g^{-1}f})}) < 2\pi$, Equations~\eqref{eq:derivative_exp} and \eqref{eq:derivative_log} and the bilinearity of ad yield
    \begin{align*}
    \bigl( \partial_1 \gamma_t (g,f) \bigr)^{-1} =\
        &\dd_{g^{-1}f}\Lt_g \\
        &\circ \frac{1}{t}\; \dd_e \Lt_{g^{-1}f} \\ &\circ \sum_{n=0}^\infty \frac{(-1)^n \ad_{\log(g^{-1}f)}^n}{(n+1)!} \\
        &\circ \sum_{k=0}^\infty \frac{(-1)^k t^k B_k \ad_{\log(g^{-1}f)}^k}{k!} \\
        &\circ \dd_{\exp(t\log(g^{-1}f))} \Lt_{\exp(t\log(g^{-1}f))^{-1}} \\
        &\circ\, \dd_{g\exp(t\log(g^{-1}f))} \Lt_{g^{-1}}.
    \end{align*}
    Since $\Lt_g \circ \Lt_f = \Lt_{gf}$ for all $g,f \in G$ the chain rule and the definition of $\gamma$ give
    \begin{align*}
        \bigl( \partial_1 \gamma_t (g,f) \bigr)^{-1} =\ &\frac{1}{t}\; \dd_e \Lt_f \\
        &\circ \sum_{n=0}^\infty \frac{(-1)^n \ad_{\log(g^{-1}f)}^n}{(n+1)!} \\
        &\circ \sum_{k=0}^\infty \frac{(-1)^k t^k B_k \ad_{\log(g^{-1}f)}^k}{k!} \\
        &\circ \dd_{\gamma(t;g,f)} \Lt_{\gamma(t;g,f)^{-1}}. \\
    \end{align*}
    Let $u \in T_{\gamma(t;g,f)}G$ and $\bar{u} := \dd_{g\exp(t\log(g^{-1}f))} \Lt_{(g\exp(t\log(g^{-1}f))^{-1}}(u)$ its left-translation at $e$. It follows with \eqref{eq:ad_dist} that
        \begin{align*}
        \bigl( \partial_1 \gamma_t (g,f) \bigr)^{-1}(u) &=\ \frac{1}{t}\; \dd_e \Lt_f \\
        &\quad \quad \circ \sum_{n=0}^\infty \frac{(-1)^n \ad_{\log(g^{-1}f)}^n}{(n+1)!} \\
        &\quad \quad \circ \sum_{k=0}^\infty \frac{(-1)^k t^k B_k \ad_{\log(g^{-1}f)}^k}{k!} \\
        &\quad \quad \circ \dd_{\gamma(t;g,f)} \Lt_{\gamma(t;g,f)^{-1}} (u) \\
        &= \frac{1}{t}\; \dd_e \Lt_f \Biggl( \sum_{n=0}^\infty \frac{(-1)^n}{(n+1)!} \ad_{\log(g^{-1}f)}^n \biggl( \sum_{k=0}^\infty \frac{(-1)^k t^k B_k}{k!}\ad_{\log(g^{-1}f)}^k(\bar{u}) \biggr) \Biggr) \\
        &= \sum_{n=0}^\infty \sum_{k=0}^\infty \frac{(-1)^{n+k} t^{k-1} B_k}{(n+1)!k!} \dd_e \Lt_f \bigl( \ad_{\log(g^{-1}f)}^{n+k}(\bar{u}) \bigr). \\
    \end{align*}
    
    We now set $m := n+k$ and collect all terms with power $m$. Since $n \ge 0$ and $n = m - k$, we finally get
    \begin{align*}
            \bigl( \partial_1 \gamma_t (g,f) \bigr)^{-1}(u) &= \sum_{m=0}^\infty \sum_{k=0}^m \frac{(-1)^{m-k+k} t^{k-1} B_k}{(m-k+1)!k!} \dd_e \Lt_f \bigl( \ad_{\log(g^{-1}f)}^m(\bar{u}) \bigr) \\
            &= \sum_{m=0}^\infty (-1)^{m} \sum_{k=0}^m \frac{t^{k-1} B_k}{(m-k+1)!k!} \dd_e \Lt_f \bigl( \ad_{\log(g^{-1}f)}^m(\bar{u}) \bigr).
    \end{align*}

\end{proof}

% Note that $\ad_v = \Ad_{\exp(v)}$ and, thus, also
% \begin{equation*}
%     \bigl( \partial_1 \gamma_t (g,f) \bigr)^{-1}(u) = \dd_e\Lt_f \circ \frac{\textnormal{I} - \Ad_{g^{-1}f}}{\textnormal{I} - \Ad_{\exp(t\log(g^{-1}f))}} (\bar u)
% \end{equation*}

The lemma implies a similar formula for the derivative with respect to the start point of the geodesic.
\begin{corollary} \label{cor:diff_1_gam}
    With the same requirements as in Lemma~\ref{lem:diff_2_gam}, we have
    \begin{equation*}
        \bigl( \partial_0 \gamma_t(g,f) \bigr)^{-1} (u)  = \sum_{m=0}^\infty C_m(1-t) \dd_e \Lt_g \bigl( \ad_{\log(f^{-1}g)}^m(\bar{u}) \bigr)
    \end{equation*}
    whenever $\rho(\ad_{\log(f^{-1}g)}) < 2\pi$.
\end{corollary}
\begin{proof}
    Using $\gamma(t;g,f) = \gamma(1-t; f,g)$, the result follows directly from Lemma~\ref{lem:diff_2_gam}. 
\end{proof}
 We are ready to define the update map.

\subsection{Local Expansion of the Update Map}

For the rest of this work, let $U\subset G$ be a normal convex neighborhood of the identity, chosen sufficiently small such that
$$
V:=\log(U)\subseteq T_eG
$$
is a neighborhood of $0$, the BCH formula (stated below) converges on $V\times V$, and
$$
\rho\left(\ad_{\log(g^{-1}f)}\right)<2\pi
\qquad\text{for all }f,g\in U.
$$
We endow $T_eG$ with an auxiliary norm to measure the size of neighborhoods. To this end, we choose an arbitrary norm $\|\cdot\|$ induced by an inner product $\langle\cdot,\cdot\rangle$. 
\bigskip

Set
\begin{equation*}
    \mbx := \begin{pmatrix}
            x_0 \\
            x_1
        \end{pmatrix} \in V^2
        \quad \text{ and } \quad
        \mby := \begin{pmatrix}
            y_1 \\
            \vdots \\
            y_N
        \end{pmatrix} \in V^{N}.
\end{equation*}
The central object of our study is the update map in logarithmic coordinates:
There exist open neighborhoods $W_{\mbx}\subseteq V^2$ and
$W_{\mby}\subseteq V^N$ of $0$, and $\lambda_0>0$, such that
\begin{align*}
    \Psi: W_{\mbx} \times W_{\mby} \times [0,\lambda_0] &\to V^2, \\
    (\mbx, \mby; \lambda) &\mapsto  \log \bigl( \Exp_{\exp(\mbx)} \bigl( \lambda \mbv \bigl(\exp(\mbx), \exp(\mby), \mbt \bigr) \bigr) \bigr)
\end{align*}
is well defined.
The map $\Psi$ represents the iterates~\eqref{alg:algorithm} in $T_eG$. We shall show that $\Psi$ has a fixed point if the data is sufficiently local. Existence, uniqueness, and convergence then follow directly from this.

Our first observation is that $\Psi$ has a very high degree of smoothness.
\begin{proposition}\label{prop:analytic}
    The mapping $\Psi$ is analytic in a neighborhood of 0.
\end{proposition}
\begin{proof}
    This follows from the fact that $\Psi$ is a composition of analytic mappings: the Lie exponential and logarithm, their inverse derivatives, addition of vectors, and multiplication by scalars.
\end{proof}

% Since $T_eG$ is
% finite-dimensional, the Lie bracket is a continuous bilinear map.
% Hence, there exists a constant $C>0$ such that
% $$
% \|[x,y]\| \le C\|x\|\|y\|
% $$
% for all $x,y\in T_eG$. Consequently,
% $$
% \|[x,y]\|
% \le \frac{C}{4}(\|x\|+\|y\|)^2,
% $$
% and thus
% $$
% \|[x,y]\|
% = \mathcal O\bigl((\|x\|+\|y\|)^2\bigr)
% \quad \text{as } (\mbx,\mby)\to(0,0).
% $$

We will make use of the Baker-Campbell-Hausdorff (BCH).  
\begin{theorem}[Series version of the BCH formula \cite{duistermaat2012lie}]
    Let $x, y \in T_eG$. If they are small enough, then the logarithm of the product $\exp(x)\exp(y)$ is given by an absolutely convergent series of nested Lie brackets of $x$ and $y$. Its expansion up to degree two reads:
    \begin{align}
        \BCH(\mbx,\mby) &:= \log(\exp(x)\exp(y)) \nonumber \\
        &= x + y + \frac{1}{2} [x, y]  + O \bigl( (\| x \| + \| y \|)^3 \bigr). \label{eq:BCH}
    \end{align}
\end{theorem}
Our conditions on $V$ ensure that the BCH formula is well-defined on $V \times V$~\cite[Section 1.7]{duistermaat2012lie}. 

The update map $\Psi$ is nonlinear because the Cartan--Schouten geodesics involve the group exponential and logarithm. Nevertheless, close to the identity, the map can be understood as a perturbation of its linearization. The purpose of the following lemma is to make this observation precise. In particular, we separate the linear contribution from the higher-order terms arising from the noncommutativity of the Lie group.
\begin{lemma} \label{lem:asymptot}
     The update map $\Psi$ is given on $W_{\mbx} \times W_{\mby} \times [0, \lambda_0]$ by
     \begin{align*}
        \Psi(\mbx, \mby, \lambda) = \begin{pmatrix}
                    x_0 + \lambda \sum_{i=1}^N (1-t_i) \bigl( y_i - (1 - t_i) x_0 - t_i x_1 \bigr) \\
                    x_1 + \lambda \sum_{i=1}^N t_i \bigl( y_i - (1 - t_i) x_0 - t_i x_1 \bigr)
                \end{pmatrix} + \widetilde{\mathcal R}(\mbx,\mby;\lambda),
    \end{align*}
    where
    $$\left\| \widetilde{\mathcal R}(\mbx,\mby;\lambda) \right\| = \mathcal{O} \Bigl( \Bigl(\|x_0\| + \| x_1\| + \sum_{i=1}^N \| y_i\| \Bigr)^2 \Bigr).$$
\end{lemma}
\begin{proof}
Set 
\begin{equation*} \label{eq:tilde_C}
          \widetilde C_m(s) := (-1)^{m} \sum_{k=0}^m \frac{s^{k+1} B_k}{(m-k+1)!k!}
\end{equation*}
and define the vectors
\begin{equation*}
    z_0 := \sum_{i=1}^N \sum_{m=0}^\infty \widetilde C_m(1-t_i) \ad_{\log(\exp(-x_1)\exp(x_0))}^{m} \bigl(\log(b_i^{-1}\exp(y_i))\bigr)
\end{equation*}
and
\begin{equation*}
    z_1 := \sum_{i=1}^N \sum_{m=0}^\infty \widetilde C_m(t_i) \ad_{\log(\exp(-x_0)\exp(x_1))}^{m} \bigl(\log(b_i^{-1}\exp(y_i))\bigr).
\end{equation*}
Combining Equation~\eqref{eq:v} with Lemma~\ref{lem:diff_2_gam} and Corollary~\ref{cor:diff_1_gam} yields
\begin{align*}
    \Psi &(\mbx, \mby, \lambda) = \begin{pmatrix}  
                \log \bigl( \exp(x_0) \exp( \lambda z_0) \bigr) \\
                \log \bigl( \exp(x_1) \exp( \lambda z_1 ) \bigr)
            \end{pmatrix}.
\end{align*}
on $W_{\mbx} \times W_{\mby} \times [0, \lambda_0]$.

The BHC formula \eqref{eq:BCH} then yields for the entries  of $\Psi$
\begin{equation*}
    \Psi_0(\mbx, \mby, \lambda) = x_0 + \lambda z_0 + \mathcal O \bigl( (\| x_0 \| + \lambda \| z_0 \|)^2 \bigr), \label{eq:psi_1}
\end{equation*}
and
\begin{equation*}
    \Psi_1(\mbx, \mby, \lambda) = x_1 + \lambda z_1 + \mathcal O \bigl( (\| x_1 \| + \lambda \| z_1 \|)^2 \bigr). \label{eq:psi_2} 
\end{equation*}

Applying the BCH formula to $z_0$ and $z_1$ gives
\begin{align*} 
    z_0 &= \sum_{i=1}^N \sum_{m=0}^\infty \widetilde C_m(1-t_i) \; \ad_{\BCH(-x_1, x_0)}^m \Bigl(\BCH \bigl(\log(b_i^{-1}), y_i \bigr) \Bigr)
\end{align*}
and
\begin{align*}
    z_1 &= \sum_{i=1}^N \sum_{m=0}^\infty \widetilde C_m(t_i) \; \ad_{\BCH(-x_0, x_1)}^m \Bigl(\BCH \bigl(\log(b_i^{-1}), y_i \bigr) \Bigr).
\end{align*}

Note that in both equations all summands with $m \ge 1$ are scalar multiples of Lie brackets involving $x_0,x_1,y_1,\dots,y_n$, \textit{with no unbracketed terms}. Therefore, and since \eqref{eq:tilde_C} gives $\widetilde C_0(s) = s$, it follows that
\begin{equation} \label{eq:z_0_asy}
    z_0 = \sum_{i=1}^N (1-t_i) \;  \BCH \bigl(\log(b_i^{-1}), y_i \bigr) +  \mathcal{O} \bigl( (\|x_0\| + \| x_1\| + \sum_{i=1}^N \| y_i\|)^2 \bigr)
\end{equation}
and 
\begin{equation} \label{eq:z_1_asy}
    z_1 = \sum_{i=1}^N t_i \; \BCH \bigl(\log(b_i^{-1}), y_i \bigr) + \mathcal{O} \bigl( (\|x_0\| + \| x_1\| + \sum_{i=1}^N \| y_i\|)^2 \bigr).
\end{equation}

It remains to compute $\BCH (\log(b_i^{-1}), y_i )$ to linear order. First,
\begin{align*}
    \log(b_i^{-1}) &= \log \Biggl( \biggl(\exp(x_0) \exp \Bigl( t_i \log \bigl(\exp(-x_0)\exp(x_1) \bigr) \Bigr) \biggr)^{-1} \Biggr)\\
    &= \log \biggl( \exp \Bigl( - t_i \log \bigl(\exp(-x_0)\exp(x_1) \bigr) \Bigr) \exp(-x_0) \biggr)\\
    &= \BCH \Bigl( - t_i \log \bigl(\exp(-x_0)\exp(x_1) \bigr), -x_0 \Bigr) \\
    &= \BCH \Bigl( - t_i \BCH \bigl(-x_0, x_1 \bigr), -x_0 \Bigr).
\end{align*}
From
\begin{align*}
    \BCH \bigl(-x_0, x_1 \bigr) =  x_1 - x_0 + \mathcal{O} \bigl( (\|x_0\| + \| x_1\|)^2 \bigr),
\end{align*}
thus follows
\begin{align*}
    \log(b_i^{-1}) &=  -t_i( x_1 - x_0) - x_0 + \mathcal{O} \bigl( (\|x_0\| + \| x_1\|)^2 \bigr) \nonumber \\
    &= -(1 - t_i) x_0 -t_i x_1 + \mathcal{O} \bigl( (\|x_0\| + \| x_1\|)^2 \bigr). \label{eq:logbi}
\end{align*}
Hence, we get
\begin{align*}
    \BCH \bigl( \log(b_i^{-1}), y_i \bigr) = y_i - (1 - t_i) x_0 -t_i x_1 + \mathcal{O} \bigl( (\|x_0\| + \| x_1\| + \| y_i\|)^2 \bigr).
\end{align*}
Inserting this into \eqref{eq:z_0_asy} and \eqref{eq:z_1_asy} gives
\begin{equation*}
    z_0 = \sum_{i=1}^N (1-t_i) \bigl( y_i - (1 - t_i) x_0 - t_i x_1 \bigr) + \mathcal{O} \bigl( (\|x_0\| + \| x_1\| + \sum_{i=1}^N \| y_i\|)^2 \bigr)
\end{equation*}
and
\begin{equation*}
    z_1 = \sum_{i=1}^N t_i \bigl( y_i - (1 - t_i) x_0 - t_i x_1 \bigr) + \mathcal{O} \bigl( (\|x_0\| + \| x_1\| + \sum_{i=1}^N \| y_i\|)^2 \bigr).
\end{equation*}

Combining these with \eqref{eq:psi_1} and \eqref{eq:psi_2} finally yields
\begin{align*}
    \Psi_0(\mbx, \mby, \lambda) = x_0 &+ \lambda \sum_{i=1}^N (1-t_i) \bigl( y_i - (1 - t_i) x_0 - t_i x_1 \bigr) \\
    &+ \mathcal{O} \bigl( (\|x_0\| + \| x_1\| + \sum_{i=1}^N \| y_i\|)^2 \bigr)
\end{align*}
and 
\begin{align*}
    \Psi_1(\mbx, \mby, \lambda) = x_1 &+ \lambda \sum_{i=1}^N t_i \bigl( y_i - (1 - t_i) x_0 - t_i x_1 \bigr) \\
    &+ \mathcal{O} \bigl( (\|x_0\| + \| x_1\| + \sum_{i=1}^N \| y_i\|)^2 \bigr).
\end{align*}
\end{proof}

\begin{remark} \label{rem:remainder}
    Remember that $\ad_x^m(y)$ is a nested Lie bracket involving only $x$ and $y$. 
    Because of this and since $\BCH(\mbx,\mby)$ is a series of nested Lie brackets involving only $x$ and $y$, the quadratic remainder term is an absolutely convergent series of nested Lie brackets. In particular, each summand is a scalar multiple of a non-trivial nested bracket that contains only $x_0, x_1, y_1, \dots, y_N$, with no unbracketed linear terms.  
\end{remark}

\subsection{Vectorization}

Before we continue, we must pick an inner product/norm on $(T_eG)^k$, $k\ge2 $: We use the euclidean inner product, that is, $$\langle (z_1, \dots, z_k), (\widetilde z_1, \dots, \widetilde z_k) \rangle_2 := \sum_{i=1}^k \langle z_i, \widetilde{z}_i \rangle$$ for $z_1, \dots z_k, \widetilde{z}_1,\dots, \widetilde{z}_k \in T_eG$. The induced $2$-norm is 
$$\| (z_1, \dots, z_k) \|_2 := \sqrt{\|z_1\|^2 + \cdots + \|z_k\|^2}.$$ 
For a linear operator $A$ on $(T_eG)^k$, we denote the induced operator norm by $\|A\|_2$.

To prepare for our fixed point approach, we start with reformulating the update map $\Psi$.
We define the linear map
\begin{align*}
    M: (T_eG)^2 &\to (T_eG)^N, \\
    x &\mapsto \begin{pmatrix}
            (1-t_1) x_0 + t_1 x_1 \\
            \vdots \\
            (1-t_N) x_0 + t_N x_1
        \end{pmatrix};
\end{align*}
its adjoint is 
\begin{align*}
    M^*: (T_eG)^N &\to (T_eG)^2, \\
    y &\mapsto \begin{pmatrix}
            \sum_{i=1}^N(1-t_i) y_i \\
            \sum_{i=1}^Nt_i y_i
        \end{pmatrix}.
\end{align*}

The map $M$ represents the linear interpolation of the two regression endpoints at the sampling times $t_i$. Thus, $Mx$ is the vector of predicted observations in logarithmic coordinates, while
$$
    y-Mx
$$
is the corresponding vector of residuals. The adjoint $M^*$ therefore aggregates these residuals into corrections for the two endpoint variables. This makes $M^*M$ the natural linear operator governing the behavior of the fixed point iteration near the identity.

We write $\Psi_{\mby}(\mbx, \lambda)$ in the rest of this work, making it more explicit that the data is seen as fixed.
The following lemma provides a useful reformulation of $\Psi_\mby$.
\begin{lemma} \label{lem:Psi_alt}
    With 
    $$a := \sum_{i=1}^N (1 - t_i)^2, \qquad b := \sum_{i=1}^N t_i(1 - t_i), \qquad \textnormal{and}  \qquad c := \sum_{i=1}^N t_i^2,$$
    let \begin{align*}
            A_\lambda: (T_eG)^2 &\to (T_eG)^2 \\
            x &\mapsto \begin{pmatrix}
        \bigl( 1 - \lambda a \bigr) x_0 - \lambda b x_1 \\
        - \lambda b x_0 +  \bigl( 1 - \lambda c \bigr) x_1
    \end{pmatrix}.
    \end{align*} 
    There exist open neighborhoods $W_{\mbx}$ and $W_{\mby}$ of 0 and $\lambda_0 >0$ such that $\Psi$ is given on $W_{\mbx} \times W_{\mby} \times [0, \lambda_0]$ by
    \begin{align*} 
        \Psi_{\mby}(\mbx, \lambda) = A_\lambda \mbx + \lambda M^*  y + \lambda \mathcal R(\mbx,\mby;\lambda),
    \end{align*}
    where 
    $$ \| R(\mbx,\mby;\lambda) \| = \mathcal{O}\bigl((\|x_0\| + \|x_1\| + \sum_{i=1}^N \|y_i\|)^2\bigr).$$
\end{lemma}
\begin{proof}
Define the residual
\begin{equation*} 
    \mbr := \mby - M \mbx \in V^{N}.
\end{equation*}
With sufficiently small open neighborhoods $W_{\mbx}$ and $W_{\mby}$ of $0$, Lemma~\ref{lem:asymptot} ensures that
\begin{align} 
    \Psi_{\mby}(\mbx, \lambda) &= \mbx + \lambda M^* \mbr + \widetilde{\mathcal R}(\mbx,\mby;\lambda)  \nonumber \\
    &= (I - \lambda M^* M) \mbx + \lambda M^*  \mby + \widetilde{\mathcal R}(\mbx,\mby;\lambda), \label{eq:vector_Psi_}
\end{align}
on $W_{\mbx} \times W_{\mby} \times [0, \lambda_0]$;
here,
$$ \| \widetilde{\mathcal R}(\mbx,\mby;\lambda) \| = \mathcal{O} \Bigl( \Bigl(\|x_0\| + \| x_1\| + \sum_{i=1}^N \| y_i\| \Bigr)^2 \Bigr).$$ 
The analyticity of $\Psi$ implies that $\widetilde{\mathcal R}$ is also analytic, in particular in $\lambda$. 
Thus 
$$\widetilde{\mathcal R}(\mbx,\mby;\lambda) = \mathcal R_0(\mbx,\mby) + \mathcal R_1(\mbx,\mby)\lambda + \mathcal R_2(\mbx,\mby)\lambda^2 + \dots$$
for $\mathcal R_i(\mbx,\mby) \in (T_eG)^2$ with $i=0,1,\dots$
By definition $\Psi_{\mby}(\mbx, 0) = \mbx$ for all $\mbx \in W_{\mbx}$ and $\mby \in W_{\mby}$.
Therefore,
$\widetilde{\mathcal R}(\mbx,\mby,0) = 0$
for all $\mbx \in W_{\mbx}$ and $\mby \in W_{\mby}$. 
Consequently, $R_0(\mbx,\mby) = 0$, and thus
$$\widetilde{\mathcal R}(\mbx,\mby;\lambda) =  \lambda \mathcal{R}(\mbx,\mby;\lambda),$$
where 
$$\mathcal{R}:= \widetilde{\mathcal R}/\lambda = \mathcal R_1(\mbx,\mby) + \mathcal R_2(\mbx,\mby)\lambda + \dots$$ 
is analytic on the whole neighborhood. 
Setting \begin{equation*}
    A_\lambda := I - \lambda M ^* M,
\end{equation*}
Equation~\eqref{eq:vector_Psi_} thus becomes
\begin{align*} 
    \Psi_{\mby}(\mbx, \lambda) = A_\lambda \mbx + \lambda M^*  \mby + \lambda \mathcal R(\mbx,\mby;\lambda).
\end{align*}

Finally, we find
\begin{align*}
    M^*M \mbx = \begin{pmatrix}
        ax_0 + bx_1 \\
        bx_0 + cx_1
    \end{pmatrix},
\end{align*}
which yields
\begin{align*}
    A_\lambda \mbx = \begin{pmatrix}
        \bigl( 1 - \lambda a \bigr) x_0 - \lambda b x_1 \\
        - \lambda b x_0 +  \bigl( 1 - \lambda c \bigr) x_1
    \end{pmatrix}.
\end{align*}
\end{proof}

Lemma~\ref{lem:Psi_alt} provides the central local decomposition of the update map. The first term,
$$
    A_\lambda \mbx=(I-\lambda M^*M) \mbx,
$$
is precisely the linearization of the fixed point iteration with respect to the endpoint variables at the identity. The second term, $\lambda M^* \mby$, describes the influence of the observed data, while $\lambda\mathcal R$ contains the nonlinear contributions resulting from the geometry of the Lie group.

\section{Local Existence, Uniqueness, and Convergence}\label{sec:existence}

We now have the prerequisites needed to prove local existence and uniqueness of the estimator. Our approach, like that of Pennec and Arsigny in~\cite{PennecArsigny2013}, is to study the fixed points of the update map $\Psi$. Because of the estimator's equivariance properties, understanding this map is sufficient. The main idea is to exploit the fact that, in a sufficiently small neighborhood of the identity, the nonlinear update map can be decomposed into a linear part and a higher-order remainder. The linear part determines the dominant behavior of the iteration, while the nonlinear terms can be controlled by restricting the data to a sufficiently small neighborhood.

In particular, we will show that, for sufficiently localized data and sufficiently small stepsize $\lambda$, the map $\Psi$ is a contraction on a suitable closed ball in logarithmic coordinates. The Banach Fixed Point Theorem then guarantees the existence of a unique fixed point in this ball and convergence of the corresponding fixed point iteration. Finally, we relate this fixed point to the estimating equations of bi-invariant geodesic regression.

Using Lemma \ref{lem:asymptot}, the subsequent analysis reduces the local behavior of the nonlinear iteration to two questions. First, under which conditions is the linear operator $A_\lambda$ a contraction? Second, can the nonlinear remainder be controlled sufficiently strongly in a small neighborhood so that it does not destroy this contraction? We address the first question next.

\subsection{Linear Contraction}

\begin{lemma} \label{eq:norm_A}
    Let $t_1,\dots,t_N \in [0,1]$ not all equal. Let further $\mu_{\textnormal{min}}$ and $\mu_{\textnormal{max}}$ be the smaller and larger eigenvalue of $M^*M$, respectively. If $0 < \lambda < 2/\mu_{\textnormal{max}}$, then
    $$\left\| A_\lambda \right\|_2 = \textnormal{max} \{|1 - \lambda \mu_{\textnormal{min}}|,  |1 - \lambda \mu_{\textnormal{max}}|\} < 1,$$
    that is, $A_\lambda$ is a contraction.
\end{lemma}
\begin{proof}
    Let $d$ be the dimension of $G$. We choose an orthonormal basis
    $$
    \mathcal{B}=(v_1,\ldots,v_d)
    $$
    of $T_eG$,
    and use, for each $(T_eG)^k$, $k\geq 1$, the induced orthonormal basis
    $$
    \mathcal{B}^{(k)}
    =
    \left\{
    e_i^{(j)} : 1\leq i\leq d,\ 1\leq j\leq k
    \right\},
    $$
    where
    $$
    e_i^{(j)}
    = (0,\ldots,0,\underbrace{v_i}_{j\text{-th position}},0,\ldots,0).
    $$
   In the following, use the representation in terms of these bases to work with matrices and calculate the norm of $A_\lambda$ from them. 
   
For the matrix representation of $M$ and its adjoint $M^*$ we need
\begin{equation*}
    \mathbf{M} := \begin{bmatrix}
            1-t_1 & t_1 \\
            \vdots & \vdots \\
            1-t_N & t_N
        \end{bmatrix} \in \mathbb{R}^{N, 2};
\end{equation*}
and $\mathbf{M}^T$, respectively.
With $\mathbf I_d$ denoting the $d$-by-$d$ identity matrix, $M$ is represented by
\begin{equation*}
    \mathcal M := \mathbf{M} \otimes \mathbf I_d = \begin{bmatrix}
        (1 - t_1) \mathbf I_d & t_1 \mathbf I_d \\
        \vdots & \vdots \\
        (1 - t_N) \mathbf I_d & t_N \mathbf I_d
    \end{bmatrix} \in \R^{Nd, 2d},
\end{equation*}
where $\otimes$ denotes the Kronecker product. Its adjoint is given by $\mathcal M^T$.

Now, we use the facts that $(\mathbf{A} \otimes \mathbf{B})^T = \mathbf{A}^T \otimes \mathbf{B}^T$ and $(\mathbf{A} \otimes \mathbf{B}) (\mathbf{C} \otimes \mathbf{D}) = \mathbf{A} \mathbf{C} \otimes \mathbf{B}\mathbf{D}$ for all matrices $\mathbf{A}, \mathbf{B}, \mathbf{C}$ and $\mathbf{D}$ for which the matrix products $\mathbf{AC}$ and $\mathbf{BD}$ are defined. Thus,
\begin{equation*}
   \mathcal M ^T \mathcal M = (\mathbf{M}^T \otimes \mathbf I_d) (\mathbf{M} \otimes \mathbf I_d) = \mathbf{M}^T \mathbf{M} \otimes \mathbf I_d \in \R^{2d, 2d}. 
\end{equation*}
Interesting for us is
\begin{equation*}
    \mathbf{H}: = \mathbf{M}^T\mathbf{M} = \begin{bmatrix}
                    \sum_{i=1}^N (1 - t_i)^2 & \sum_{i=1}^N t_i(1 - t_i) \\
                    \sum_{i=1}^N t_i(1 - t_i) & \sum_{i=1}^N t_i^2
                \end{bmatrix} \in \R^{2, 2}.
\end{equation*}
Note that with 
\begin{equation*}
    \mathbf{h}_i := \begin{bmatrix}
        1-t_i \\
        t_i
    \end{bmatrix}, \quad i=1,\dots,N,
\end{equation*}
we have
\begin{equation*}
    \mathbf{H} = \sum_{i=1}^N \mathbf{h}_i \mathbf{h}_i^T.
\end{equation*}
Clearly, $\mathbf{H}$ is symmetric positive-semidefinite. Moreover, an eigenvalue of $\mathbf{H}$ is 0 if and only if there is a vector $\mathbf x \neq 0$ orthogonal to all $\mathbf{h}_i$. Since the latter all lie at distinct points along the line segment from $\begin{bmatrix}
    1 & 0
\end{bmatrix}^T$ to $\begin{bmatrix}
    0 & 1
\end{bmatrix}^T$, two vectors
$\mathbf{h}_i$ and $\mathbf{h}_j$ are linearly independent if and only if $t_i \ne t_j$. Consequently, if there are $i,j \in \{1,\dots,N\}$ with $t_i \ne t_j$, then $\mathbf H$ has two positive eigenvalues $0 < \mu_{\textnormal{min}} \le \mu_{\textnormal{max}}$.\footnote{Both can be computed with standard formulas; they are
\begin{equation*}
    \mu_{\textnormal{max},\textnormal{min}} = \frac{\sum_{i=1}^N (1 -2t_i + 2t_i^2)}{2} \pm \sqrt{ \frac{\Bigl(\sum_{i=1}^N (1 - 2t_i) \Bigr)^2}{4} + \Biggl(\sum_{i=1}^N t_i(1-t_i) \Biggr)^2}.
\end{equation*}}

In coordinates, the map $A_\lambda$ is given by the matrix
\begin{equation*}
    \mathbf{A}_{\lambda} = (\mathbf I_{2n} - \lambda \mathbf H \otimes \mathbf I_d).
\end{equation*}
This matrix is symmetric because its components are all symmetric. Remember that the eigenvalues of the tensor product $ \mathbf{A} \otimes \mathbf{B}$ are formed by multiplying every eigenvalue of $\mathbf{A}$ by every eigenvalue of $\mathbf{B}$. Therefore, the eigenvalues of $\mathbf{A}_\lambda$ are 
\begin{equation*}
    1 - \lambda \mu_{\textnormal{min}} \quad \textnormal{and} \quad 1 - \lambda \mu_{\textnormal{max}},
\end{equation*}
both with multiplicity $n$. 
The 2-norm of $A_\lambda$ is thus
\begin{equation*}
    \| A_\lambda \|_2 = \textnormal{max} \{|1 - \lambda \mu_{\textnormal{min}}|,  |1 - \lambda \mu_{\textnormal{max}}|\}.
\end{equation*}
Now, $A_\lambda$ is a contraction~\cite{kreyszig1991introductory} if and only if $-1 < 1-\lambda \mu_{\textnormal{max}} < 1$.
This is the case when 
$$0 < \lambda < \frac{2}{\mu_{\textnormal{max}}}.$$
\end{proof}

The previous lemma shows that $A_\lambda$ is a contraction whenever

$$
    0<\lambda<\frac{2}{\mu_{\mathrm{max}}}.
$$

For the remainder of the proof, we impose the slightly stronger restriction
$$
    0<\lambda\leq\frac{1}{\mu_{\mathrm{max}}},
$$
which has the convenient consequence that both eigenvalues of $A_\lambda$ are nonnegative. Hence,
$$
    \|A_\lambda\|_2=1-\lambda\mu_{\mathrm{min}}.
$$

This explicit expression will allow us to compare the contraction induced by the linear part with the contribution of the nonlinear remainder.

The expansion in Lemma~\ref{lem:Psi_alt} separates the behavior of the iteration into a linear part and a nonlinear remainder. The linear part is governed by the map $M^*M$. If the sampling parameters $t_i$ are not all identical, then $M^*M$ is positive definite. Consequently, as shown above, the linear operator $A_\lambda$ is a contraction for sufficiently small $\lambda>0$.

This observation alone, however, does not imply that the full nonlinear map $\Psi_\mby$ is a contraction. We must also ensure that the iterates remain in the neighborhood on which the local expansion is valid. Moreover, the contribution of the data must be small enough that the contraction of the linear part is not overwhelmed by the data term.

\subsection{Invariant Ball}
In the following, we denote the 2-norm ball of radius $R$ in $(T_eG)^2$ by
$$
B_R^2(0):=
\left\{
\mbz \in(T_eG)^2: \| \mbz \|_2 \le R
\right\}.
$$
The next lemma establishes that, for sufficiently localized data, $\Psi_\mby$ maps such a neighborhood of the origin into itself.
\begin{lemma} \label{lem:selfmap}
    Let $t_1,\dots,t_N \in[0,1]$ not all equal. Further, set
    $$
        \lambda^* := \textnormal{min}\left(\lambda_0, \frac{1}{\mu_{\textnormal{max}}} \right)  \qquad \textnormal{and} \qquad
        \alpha^* := \frac{\sqrt{N \mu_{\textnormal{max}}}}{\mu_{\textnormal{min}}},
    $$
    and let
    $$\alpha > \alpha^*.$$
    Then there exists $r_s(\alpha)>0$ such that the following holds:
    For every $r\in(0,r_s(\alpha)]$, set
    $$
    R:= \alpha r.
    $$
    Then, for every dataset
    $(y_i)_{i=1}^N$ satisfying
    $$
    \| y_i \| \le r, \qquad i=1,\dots,N,
    $$
    and every $0<\lambda\le \lambda^*$, we have
    $$\Psi_\mby(B^2_R(0)) \subseteq B^2_R(0).$$
\end{lemma}
\begin{proof}
Lemma~\ref{lem:Psi_alt} yields 
\begin{align*}
    \| \Psi_\mby (\mbx, \lambda) \|_2  
    &= \biggl\| A_\lambda \mbx + \lambda M^*  y + \lambda \mathcal R(\mbx,\mby;\lambda) \biggr\|_2 \\   
    &\le \| A_\lambda \|_2\ \| \mbx \|_2 + \lambda \| M^* \|_2\ \| \mby \|_2 + \lambda \|\mathcal{R}(\mbx,\mby;\lambda)\|_2 \\
    &\le (1-\lambda \mu_{\textnormal{min}}) \| \mbx \|_2 + \lambda \| M^* \|_2\ \| \mby \|_2 + \lambda \|\mathcal{R}(\mbx,\mby;\lambda)\|_2.
\end{align*}
on a neighborhood $W_{\mbx} \times W_{\mby} \times [0, \lambda_0]$.
Hence, when $\lambda \le \textnormal{min}(\lambda_0, 1/\mu_{\textnormal{max}})$, we find
\begin{align*}
    \| \Psi_\mby (\mbx, \lambda) \|_2 
    &\le (1-\lambda \mu_{\textnormal{min}}) \| \mbx \|_2 + \lambda \| M^* \|_2\ \| \mby \|_2 + \lambda \|\mathcal{R}(\mbx,\mby;\lambda)\|_2.
\end{align*}

Since $\| M^* \|_2^2$ is the largest eigenvalue of $M^*M$ (see~\cite[Chapter 7]{axler2024linear}), we get $\| M^* \|_2 =\sqrt{\mu_{\textnormal{max}}}.$
Therefore,
\begin{align*}
    \| \Psi_\mby (\mbx, \lambda) \|_2  
    &\le (1-\lambda \mu_{\textnormal{min}})  \| \mbx \|_2 + \lambda  \sqrt{\mu_{\textnormal{max}}}\ \| \mby \|_2 + \lambda \| \mathcal{R}(\mbx,\mby;\lambda) \|_2.
\end{align*}

If $x \in B^2_{\alpha r}(0)$, then Cauchy-Schwartz implies $\|x_0\| + \| x_1\| \le \sqrt{2} \| \mbx \|_2 \le \sqrt{2} \alpha r$. Similarly, when $ \|y_i\| \le r$ for all $i=1,\dots,N$, we have $\sum_{i=1}^N \| y_i\| \le \sqrt{N} \| \mby \|_2 \le Nr$.
Hence, under our assumptions, we find
$$ \|x_0\| + \| x_1\| + \sum_{i=1}^N \| y_i\| \le \left(\sqrt{2} \alpha + N \right) r.$$
Consequently, for some $C > 0$,
\begin{align*}
    \| \Psi_\mby (\mbx, \lambda) \|_2  
    &\le (1-\lambda \mu_{\textnormal{min}})  \alpha r + \lambda  \sqrt{N\mu_{\textnormal{max}}} r + \lambda C (\sqrt{2} \alpha + N )^2 r^2.
\end{align*}

For $\Psi_\mby$ to map $B^2_R(0)$ into itself, we need $\| \Psi_\mby (\mbx, \lambda) \|_2 \le R = \alpha r$; this is true if,
\begin{align} \label{eq:ineq}
    (1-\lambda \mu_{\textnormal{min}})  \alpha r + \lambda  \sqrt{N\mu_{\textnormal{max}}} r + \lambda C (\sqrt{2} \alpha + N )^2 r^2 \le \alpha r.
\end{align}
When $\lambda > 0$, this is equivalent to
\begin{align*}
    \mu_{\textnormal{min}} \alpha \ge +  \sqrt{N\mu_{\textnormal{max}}} + C (\sqrt{2} \alpha + N )^2 r.
\end{align*}
So, if $\alpha > \sqrt{N \mu_{\textnormal{max}}} / \mu_{\textnormal{min}}$, then there is a positive margin
$$\delta_\alpha := \mu_{\textnormal{min}} \alpha - \sqrt{N\mu_{\textnormal{max}}}$$ that can absorb the quadratic term.
Indeed, let $\widetilde{r}>0$ be sufficiently small such that
$B^2_{\alpha r}(0) \subseteq W_{\mbx}$ and $\{(y_1,\dots,y_N) \in (T_eG)^N: \textnormal{max}(\| y_i \|) \le r \}  \subseteq W_{\mby}$
hold for all $r\le \widetilde{r}$. Choose $$r_s(\alpha) := \textnormal{min} \left ( \frac{\delta_\alpha}{C(\sqrt{2} \alpha + N )^2},\widetilde{r} \right).$$
Then, \eqref{eq:ineq} is true for every $r\in(0,r_s(\alpha)]$, and $\Psi_\mby$ maps $B^2_R(0)$ into itself.

\end{proof}

The self-map property is the first essential ingredient in the fixed point argument. It guarantees that, once an iterate has entered the chosen ball, all subsequent iterates remain in the same ball. Consequently, the local expansion from Lemma~\ref{lem:Psi_alt} remains applicable throughout the iteration.

\begin{remark}
The condition $\alpha>\alpha^*$ has a natural interpretation. The radius $\alpha r$ of the ball for the unknown endpoints must be sufficiently large compared with the size $r$ of the data. At leading order, the contraction of the endpoint variables contributes the term $\mu_{\mathrm{min}}\alpha r$, whereas the data contributes a term of size $\sqrt{N\mu_{\mathrm{max}}},r$. Choosing
$$\alpha>\frac{\sqrt{N\mu_{\mathrm{max}}}}{\mu_{\mathrm{min}}}
$$
ensures that the linear contraction provides a positive margin which can absorb the quadratic remainder when $r$ is sufficiently small.
\end{remark}

\subsection{Contraction of the Nonlinear Map}

It remains to show that the nonlinear remainder is sufficiently small in the Lipschitz sense. This is where the fact that $\mathcal R$ has no terms of degree less than two becomes important.
The following lemma uses this fact to show that the remainder itself and its derivative with respect to the endpoint variables vanish at the origin. Furthermore, it uses this to bound derivative in a neighborhood of 0.

\begin{lemma} \label{lem:dR}
    The remainder $\mathcal R$ fulfills
    $$
    \mathcal R(0,0,\lambda)=0
    \qquad \textnormal{and} \qquad
    \partial_0\mathcal R(0,0,\lambda)=0
    $$
    for all $\lambda \in [0,\lambda_0]$. Furthermore, for $C > 0$ independent of $\lambda$ there are compact neighborhoods $\mathcal C_\mbx$ and $\mathcal C_\mby$ of 0, such that
    $$
    \left\| \partial_0\mathcal R(\mbx,\mby;\lambda) 
    \right\|_2
    \le C (\|\mbx\|_2 + \|\mby\|_2)
    $$
    on $\mathcal C_\mbx \times \mathcal C_\mby \times [0, \lambda_0]$.
\end{lemma}
\begin{proof}
As discussed in Remark~\ref{rem:remainder}, all summands in the development of $\mathcal R$ are given by nested Lie brackets containing only $x_0,x_1, y_1,\dots,y_N$. Differentiating with respect to $(x_0,x_1)$ preserves this structure: for example, $\partial_0[x_0,x_1](u)=[u,x_1]$. Consequently, every summand in the development of $\partial_0\mathcal R$ contains at least one factor that vanishes when $x_0=x_1=y_1=\cdots=y_N=0$. Hence, setting $x_0=x_1=y_1=\cdots=y_N=0$ yields $\mathcal R(0,0,\lambda)=0$ and $\partial_0\mathcal R(0,0,\lambda)=0$ for every $\lambda \in [0, \lambda_0]$.    

It remains to show that the norm of the derivative can be bounded in terms of $x$ and $y$ in a compact local neighborhood.
Let
$$
F(\mbx,\mby;\lambda):=\partial_0\mathcal R(\mbx,\mby;\lambda).
$$
For fixed $\lambda$, consider the curve
$$
\ell(s):=(s\mbx,s\mby),\qquad s\in[0,1].
$$
Then
$$
F(\mbx,\mby;\lambda)-F(0,0;\lambda)
=
\int_0^1 \frac{\dd}{\dd s}F(\ell(s); \lambda)\,\dd s.
$$
By the chain rule,
$$
\frac{\dd}{\dd s}F(\ell(s); \lambda) = \frac{\dd}{\dd s}F(s\mbx,s\mby;\lambda)
=
\partial_{(1,2)}F(s\mbx,s\mby;\lambda)
\begin{pmatrix}
\mbx\\
\mby
\end{pmatrix}.
$$
Since
$$
F(0,0;\lambda)=0,
$$
we obtain
$$
F(\mbx,\mby;\lambda)
=
\int_0^1
\partial_{(1,2)}F(s\mbx,s\mby;\lambda)
\begin{pmatrix}
\mbx\\
\mby
\end{pmatrix}
\,\dd s.
$$
Taking norms gives
$$
\|F(\mbx,\mby;\lambda)\|_2
\le
\int_0^1
\left\|
\partial_{(1,2)}F(s\mbx,s\mby;\lambda)
\right\|_2
\left\|
\begin{pmatrix}
\mbx\\
\mby
\end{pmatrix}
\right\|_2
\,\dd s.
$$
If the compact neighborhood $\mathcal C_\mbx \times \mathcal C_\mby$ is small enough, then the analyticity of $F$ allows us to bound the norm of the derivative:
$$
\left\|
\partial_{(1,2)}F(s\mbx,s\mby;\lambda)
\right\|_2
\le C
$$
on $\mathcal C_\mbx \times \mathcal C_\mby \times [0, 1]$.
Since $F(\mbx,\mby;\lambda) = \partial_0\mathcal R(\mbx,\mby;\lambda)$, this gives
$$
    \left\| \partial_0\mathcal R(\mbx,\mby;\lambda) 
    \right\|_2 \le C (\|\mbx\|_2 + \|\mby\|_2).
$$

\end{proof}

Knowing that $\partial_0 \mathcal R(0,0, \lambda) = 0$,
the Lipschitz constant of the remainder can be made arbitrarily small by restricting the data and the iterates to a sufficiently small neighborhood. This is the missing piece to show that $\Psi$ is locally a contraction.

\begin{lemma} \label{lem:contaction}
Let $t_1,\dots,t_N \in[0,1]$ not all equal. Let further $\lambda^*$ and $\alpha^*$ be as in Lemma~\ref{lem:selfmap},
and $\alpha > \alpha^*$. Then there exist $r_c(\alpha)>0$ such that the following holds:
For every $r\in(0,r_c(\alpha)]$, set
$$
R:= \alpha r.
$$ 
Then, for every data set
$(y_i)_{i=1}^N$ satisfying
$$
\| y_i \| \le r, \qquad i=1,\dots,N,
$$
and every $0<\lambda\le\lambda^*$, the map
$$
\Psi_\mby:B^2_R(0)\to (T_eG)^2
$$
is a contraction on $B^2_R(0)$ with respect to the 2-norm.
\end{lemma}
\begin{proof}
Lemma~\ref{lem:Psi_alt} ensures that there is a neighborhood $W_{\mbx} \times W_{\mby} \times [0, \lambda_0]$ on which
\begin{align} \label{eq:psi_difference}
    \Psi_{\mby}(\mbx, \lambda) - \Psi_{\mby}(\widetilde{\mbx}, \lambda) = A_\lambda (\mbx - \widetilde{\mbx}) + \lambda \big( \mathcal R(\mbx,\mby;\lambda) - \mathcal R(\widetilde{\mbx},\mby;\lambda) \bigr).
\end{align}
for $\mbx, \widetilde{\mbx} \in W_{\mbx}$.

If $\lambda < 1/\mu_{\textnormal{max}}$ We know from Lemma~\eqref{eq:norm_A} that $A_\lambda$ is a contraction with contraction factor $1 - \lambda \mu_{\textnormal{min}} < 1$, that is,
\begin{align} \label{eq:contraction_A}
\| A_\lambda (\mbx - \widetilde{\mbx}) \|_2 \le (1 - \lambda \mu_{\textnormal{min}}) \| \mbx - \widetilde{\mbx} \|_2.
\end{align}
To combine \eqref{eq:psi_difference} and \eqref{eq:contraction_A}, we must assume $0 < \lambda \le \lambda^* = \textnormal{min}(\lambda_0, 1/\mu_{\textnormal{max}})$.

It remains to bound the rest.
Lemma~\ref{lem:dR} gives 
\begin{equation*} \label{eq:dr_bound}
        \|\partial_0\mathcal R(\mbx,\mby;\lambda)\|_2
    \leq C'(\|\mbx\|_2+\|\mby\|_2)
\end{equation*}
on a smaller compact neighborhood $\mathcal{C}_\mbx \times \mathcal{C}_\mby \times [0, \lambda^*]$.
The constant~$C'$ can be chosen uniformly for $\lambda\in[0,\lambda^*]$.\footnote{We label the constant $C'$ to distinguish it from the constant that appeard in the proof of Lemma~\ref{lem:selfmap}.} This uniformity is important because the subsequent contraction estimate is required to hold simultaneously for all sufficiently small stepsizes.

From our assumptions, we now get
\begin{align} 
    \|\partial_0 \mathcal R(\mbx,\mby;\lambda)\|_2 \le C'(R + \sqrt{N}r), \label{eq:bound_D1_R}
\end{align}
where $r > 0$ is sufficiently small such that $\mbx \in \mathcal{C}_\mbx$ and $\mby \in \mathcal{C}_{\mby}$.

We now want to bound 
$$\mathcal{R}(\mbx, \mby;\lambda) - \mathcal{R}(\widetilde{\mbx}, \mby;\lambda)$$
on $\mathcal{C}_x$.
Let 
$$\ell(s) = \widetilde{\mbx} + s(\mbx - \widetilde{\mbx}), \quad s \in [0,1]$$
be the straight line between $\mbx$ and $\widetilde{\mbx}$. It does not leave the domain because $\mathcal{C}_\mbx$ is convex.
Set 
$$F(s) = \mathcal{R}(\ell(s), \mby; \lambda).$$
Then,
$$F(1) - F(0) = \mathcal{R}(\mbx,\mby;\lambda) - \mathcal{R}(\widetilde{\mbx},\mby;\lambda).$$
We also have
$$F(1) - F(0) = \int_0^1 F'(s) \dd s.$$
Because 
$$F'(s) = \partial_0 \mathcal{R}(\ell(s), \mby;\lambda)(\mbx - \widetilde{\mbx}),$$
it follows that
$$\mathcal{R}(\mbx,\mby;\lambda) - \mathcal{R}(\widetilde{\mbx},\mby;\lambda) = \int_0^1 \partial_0 \mathcal{R}(\ell(s), \mby;\lambda)(\mbx - \widetilde{\mbx}) \dd s.$$
Consequently,
$$\| \mathcal{R}(\mbx,\mby;\lambda) - \mathcal{R}(\widetilde{\mbx},\mby;\lambda) \|_2 \le \int_0^1 \| \partial_0 \mathcal{R}(\ell(s), y;\lambda) \|_2 \; \|(\mbx - \widetilde{\mbx})\|_2 \dd s.$$
Equation~\eqref{eq:bound_D1_R} then yields
\begin{align}
    \left\| \mathcal{R}(\mbx,\mby;\lambda) - \mathcal{R}(\widetilde{\mbx},\mby;\lambda) \right\|_2 &\le \int_0^1 C' (R + \sqrt{N}r) \|(\mbx - \widetilde{\mbx})\|_2 \dd s \nonumber \\
    &= C'(R + \sqrt{N}r) \|(\mbx - \widetilde{\mbx})\|_2. \label{eq:lip_R}
\end{align}
on $\mathcal{C}_x \times \mathcal{C}_y \times [0, \lambda^*]$.
\bigskip

We can finally investigate 
$$\Psi_{\mby}(\mbx,\lambda) - \Psi_{\mby}(\widetilde{\mbx},\lambda).$$
Equation \eqref{eq:psi_difference} yields
\begin{align*}
    \| \Psi_{\mby}(\mbx, \lambda) - \Psi_{\mby}(\widetilde{\mbx}, \lambda) \|_2 &\le \| A_\lambda(\mbx - \widetilde{\mbx}) \|_2 + \lambda \| \mathcal R(\mbx,\mby;\lambda) - \mathcal R(\widetilde{\mbx},\mby;\lambda) \|_2
\end{align*}
on $W_{\mbx} \times W_{\mby} \times [0, \lambda_0]$.
Applying \eqref{eq:contraction_A} and \eqref{eq:lip_R} to this inequality gives
\begin{align*}
    \| \Psi_{\mby}(\mbx, \lambda) - \Psi_{\mby}(\widetilde{\mbx}, \lambda) \|_2 &\le \bigl( 1 - \lambda \mu_{\textnormal{min}}  + \lambda C'(R + \sqrt{N}r) \bigr) \| \mbx - \widetilde{\mbx} \|_2.
\end{align*}
on the (possibly smaller) domain $\mathcal{C}_x \times \mathcal{C}_y \times [0, \lambda^*]$.
Consequently,
$\Psi_{\mby}$ is a contraction on $B^2_R(0) \subseteq \mathcal{C}_x$ when 
$$  1 - \lambda \mu_{\textnormal{min}}  + \lambda C'(R + \sqrt{N}r) < 1,$$
which is equivalent to
$$ C' \left(R + \sqrt{N} r \right) < \mu_{\textnormal{min}}.$$
Using $R = \alpha r$, we obtain
$$ C' \left(\alpha + \sqrt{N} \right) r < \mu_{\textnormal{min}},$$
or, equivalently,
$$ r < \frac{\mu_{\textnormal{min}}}{C' \left(\alpha + \sqrt{N} \right)}.$$
With $\widetilde r > 0$ fulfilling $B^2_{\alpha \widetilde r}(0) \subseteq \mathcal{C}_x$ as well as $\{(y_1,\dots,y_N) \in (T_eG)^N: \textnormal{max}(\| y_i \|) \le r \} \subseteq \mathcal{C}_y$ and $\epsilon << 1$ set 
$$r_c(\alpha) := (1- \epsilon) \; \textnormal{min}\left(\frac{\mu_{\textnormal{min}}}{C' \left(\alpha + \sqrt{N} \right)}, \widetilde r \right).$$
Then, $\Psi_\mby$ is a contraction on $B^2_{\alpha r}(0)$ for all $0 < r \le r_c(\alpha)$.
\end{proof}

The contraction lemma combines the two effects identified above. The linear part contracts with factor $1-\lambda\mu_{\mathrm{min}}$, while the nonlinear remainder contributes at most
$$
    \lambda C'(\alpha+\sqrt N)r.
$$
Thus, for sufficiently small $r$, the nonlinear term reduces the contraction margin, while $\lambda$ determines the overall size of both the linear correction and this nonlinear perturbation. Once $r$ is sufficiently small, the same neighborhood works uniformly for all $\lambda\in(0,\lambda^*]$.

\begin{remark}
The proof also provides an upper bound for the contraction factor of $\Psi$. Under the assumptions of Lemma~\ref{lem:contaction}, we may take
\begin{equation} \label{eq:eta}
    K(\alpha, r, \lambda) = 1-\lambda\eta(\alpha, r),
\qquad
\eta(\alpha, r) := \mu_{\textnormal{min}}-C'(\alpha+\sqrt{N})r>0.
\end{equation}
Here, $C'$ is the constant in the bound on $\partial_0\mathcal R$. While the dependence of the contraction factor on $\lambda$, $\mu_{\textnormal{min}}$, $\alpha$, $N$, and $r$ is explicit, obtaining a numerical convergence-rate estimate requires a computable upper bound for $C'$.
\end{remark}

\subsection{Existence, Uniqueness, and Convergence}

Remember Banach's Fixed Point Theorem:
\begin{theorem}[Banach Fixed Point Theorem~\cite{kreyszig1991introductory}]   
Let $(E,d)$ be a complete metric space and $\varphi: E \to E$ a contraction, that is, $d(\varphi(x), \varphi(y)) \le c\, d(x, y)$ for all $x,y \in E$ and some $0 \le c < 1$. Then, $\varphi$ has a unique fixed point $p \in E$. Furthermore, the sequence $x_{n+1} := \varphi(\mbx_n)$ converges at least with $c$-linear speed to $p$ for any initial value in $E$.   
\end{theorem}

We can now combine the preceding estimates. The self-map lemma provides a closed ball on which the iteration is well-defined and remains inside the local coordinate neighborhood, while the contraction lemma shows that the restriction of $\Psi_\mby$ to this ball has a Lipschitz constant strictly smaller than one. Since the ball is a closed subset of the finite-dimensional space $(T_eG)^2$, it is complete. Banach's Fixed Point Theorem therefore yields a unique fixed point $\hat x$ in this ball, and the iterates generated by $\Psi_\mby$ converge to $\hat x$ from every initial point in the ball.

It remains to identify this fixed point with the bi-invariant geodesic regression estimator. The fixed point equation is formulated in logarithmic coordinates, whereas the estimator is characterized by the vanishing of the two velocity terms defining the regression geodesic. Because the exponential map is a diffeomorphism on the chosen normal convex neighborhood, the fixed point equation can be transferred back to the group. Since $\lambda>0$, the fixed point equation then implies that both velocity terms vanish. Hence the fixed point corresponds exactly to a bi-invariant estimator.

The following theorem summarizes these conclusions and makes the local nature of the result explicit.

\begin{theorem} [Local result for data concentrated at the identity]\label{thm:existance_and_uniquness}
Let $t_1,\dots,t_N\in[0,1]$ be not all equal. Let $\lambda^*$ and $\alpha^*$ be as in Lemma~\ref{lem:selfmap}, and let $\alpha>\alpha^*$. Then there exists $r^*>0$ such that the following holds:
For every $r\in(0,r^*]$, set
$$
    R:=\alpha r.
$$
Then, for every data set $(y_i)_{i=1}^N$ satisfying
$$
    \|y_i\|\le r,\qquad i=1,\dots,N,
$$
and every stepsize $0<\lambda\le\lambda^*$, the bi-invariant geodesic regression problem admits a unique estimator
$$
    \vartheta^{\mathrm{BI}}
    =
    \exp(\hat \mbx)
$$
in the local neighborhood $\exp(B_{\alpha r}^2(0))$.

Moreover, there exists $\eta:= \eta(\alpha, r)$ given in Equation \eqref{eq:eta}, independent of $\lambda$, such that the fixed point iteration
$$
    \mbx_{n+1}:=\Psi_\mby(\mbx_n,\lambda)
$$
converges to $\hat \mbx=(\hat \mbx_0,\hat \mbx_1)$ for every
$$
    \mbx_0\in B_{\alpha r}^2(0),
$$
with the estimate
$$
    \|\mbx_n-\hat \mbx\|_2
    \le
    (1-\lambda\eta)^n
    \|\mbx_0-\hat \mbx\|_2.
$$
In particular, the iteration converges at least linearly in logarithmic coordinates. Consequently, the corresponding group-valued iterates
$$
    \mbg_n:=\exp(\mbx_n)
$$
converge linearly to
$$
    \vartheta^{\mathrm{BI}}=\exp(\hat \mbx).
$$
\end{theorem}

% \begin{theorem}   
%     Let $t_1,\dots,t_N \in[0,1]$ not all equal. Let further $\lambda^*$ and $\alpha^*$ be as in Lemma~\ref{lem:selfmap},
%     and $\alpha > \alpha^*$. Then there exist $r^*>0$ such that the following holds:
%     For every $r\in(0,r^*]$, set
%     $$
%     R:= \alpha r.
%     $$ 
%     Then, for every data set
%     $(y_i)_{i=1}^N$ satisfying
%     $$
%     \| y_i \| \le r, \qquad i=1,\dots,N,
%     $$
%     and every $0<\lambda\le\lambda^*$, there exists a unique bi-invariant estimator $\mbthetabi$ for geodesic regression in that local neighborhood. Furthermore, the Algorithm \textcolor{red}{proposed by Schade et al.} converges at least linearly to $\mbthetabi$ when the stepsize is sufficiently small. 
% \end{theorem}
\begin{proof}
    Combining Lemmas \ref{lem:selfmap} and \ref{lem:contaction}, we see that under our assumptions there exists a positive number $r^*(\alpha) := \textnormal{min}(r_c(\alpha), r_s(\alpha)) > 0$ such that for all $0 < \lambda \le \lambda^*$ the update map
    $\Psi_{\mby}$ is a contractive self-map on $B^2_{\alpha r}(0) \subseteq V^2$.
    In this case, Banach's Fixed Point Theorem applies: $\Psi_{\mby}$ has a unique fixed point $\hat \mbx \in B^2_{\alpha r}(0)$.
    But since $\Exp$ is invertible in $U$ and $\lambda > 0$, this is only possible if 
    $$\mbv(\exp(\hat \mbx), \mbf, \mbt) = 0.$$ Therefore, $\hat \mbg := \exp(\hat \mbx)$ consists of the starting and end points of a bi-invariant estimator for geodesic regression for the given data. 
    
    Banach's theorem also implies that the sequence $\mbx_{n+1}:= \Psi_\mby(\mbx_n, \lambda)$ converges at least with linear speed to $\hat \mbx$ for all initial values in $B^2_{\alpha r}(0)$. 
    Since the exponential map is smooth, it is Lipschitz on the compact neighborhood under consideration; hence linear convergence in logarithmic coordinates implies linear convergence of the corresponding group-valued iterates.
    Therefore, the sequence $\mbg_{n}:= \exp(\mbx_n)$ provided by Algorithm~\ref{alg:algorithm} converges at least with linear speed to $\hat \mbg$ for all initial values in $\exp(B^2_{\alpha r}(0))$.
\end{proof}

The theorem shows that the estimator is locally well posed: Sufficiently concentrated data leads to a unique estimator in a neighborhood of the identity, and the proposed fixed point iteration converges to this estimator. The assumptions are local because the proof relies on logarithmic coordinates and on controlling the nonlinear terms in a neighborhood where the Cartan–Schouten geodesics admit the required analytic representation.

\begin{remark}[Convergence behavior]
The convergence estimate in Theorem~\ref{thm:existance_and_uniquness} provides a quantitative characterization of the behavior of the fixed point iteration. In particular,
% if $\hat{x}$ denotes the unique fixed point and $x_k$ the $k$-th iterate,
$$
\|\mbx_n-\hat{\mbx}\|_2
\leq
(1-\lambda\eta)^n
\|\mbx_0-\hat{\mbx}\|_2.
$$
Thus, the iteration converges at least linearly, with contraction factor $1-\lambda\eta<1$. For fixed $\eta$, increasing $\lambda$ improves this guaranteed convergence rate, since the contraction factor decreases linearly with $\lambda$. 
Conversely, for small $\lambda$,
$$
(1-\lambda\eta)^n
\leq
e^{-\lambda\eta n},
$$
so that the number of iterations required to reduce the initial error by a prescribed factor grows on the order of $1/\lambda$. 
The estimate should be understood as a theoretical upper bound: The actual convergence may be faster, and its numerical value depends on the constant $C'$ entering the estimate for the nonlinear remainder.
\end{remark}

\subsection{Equivariance and General Local Data}

As our final result, we use the estimator's equivariance to transfer the result to sufficiently concentrated data in a general location.
\begin{corollary}[Local result for general concentrated data]
Under the assumptions of Theorem~\ref{thm:existance_and_uniquness}, let
$h\in G$ and let $f_1,\dots,f_N\in G$ be such that
$$
h^{-1}f_i\in U,
\qquad
\bigl\|\log(h^{-1}f_i)\bigr\|\le r,
\qquad i=1,\dots,N,
$$
for some $r\in(0,r^*(\alpha)]$.
Then there exists a unique bi-invariant estimator
$\vartheta^{\mathrm{BI}}$ in the translated local neighborhood
$$
h\exp\bigl(B_{\alpha r}^2(0)\bigr)
$$
and the fixed point iteration of Theorem~\ref{thm:existance_and_uniquness}
converges at least linearly to $\vartheta^{\mathrm{BI}}$ for all initial values in this neighborhood.

More precisely, if
$$
\tilde f_i:=h^{-1}f_i
$$
and $\tilde \mbx_n$ denotes the iterates for the translated data, then
$$
\tilde \mbx_n\longrightarrow \hat \mbx
$$
at least linearly, and the corresponding iterates for the original data satisfy
$$
\mbg_n=\mbg\exp(\tilde \mbx_n)
\longrightarrow
\mbg\exp(\hat \mbx)
=\vartheta^{\mathrm{BI}}.
$$
The same convergence rate as in Theorem~\ref{thm:existance_and_uniquness}
is obtained.
\end{corollary}
\begin{proof}
    It follows directly from Proposition~\ref{prop:equivariance_algorithm} and Theorem~\ref{thm:existance_and_uniquness}.
\end{proof}
With this, we have established the local existence and uniqueness of the estimator and the local linear convergence of the algorithm in the whole group $G$.

\printbibliography

\end{document}